\documentclass[11pt]{amsart}
\usepackage[a4paper,margin=2.5cm]{geometry}
\usepackage{hyperref,mathtools,amsmath,amsthm,amssymb,amscd,amsfonts,color,tikz-cd,pgfplots,diagbox,dsfont}
\usetikzlibrary{arrows}

\theoremstyle{plain}
\newtheorem{theorem}{Theorem}[section]
\newtheorem{proposition}[theorem]{Proposition}
\newtheorem{lemma}[theorem]{Lemma}

\numberwithin{equation}{section}

\theoremstyle{definition}

\newtheorem{definition}[theorem]{Definition}
\newtheorem{remark}[theorem]{Remark}
\newtheorem{example}[theorem]{Example}

\newcommand{\cC}{\mathcal{C}}
\newcommand{\cT}{\mathcal{T}}

\newcommand{\cF}{\mathcal{F}}
\newcommand{\cR}{\mathcal{R}}

\newcommand{\cX}{\mathcal{X}}
\newcommand{\cG}{\mathcal{G}}

\newcommand{\cH}{\mathcal{H}}
\newcommand{\fS}{\mathfrak{S}}
\newcommand{\fA}{\mathfrak{A}}
\newcommand{\fI}{\mathfrak{I}}

\newcommand{\Q}{\mathbb{Q}}

\newcommand{\R}{\mathbb{R}}
\newcommand{\Z}{\mathbb{Z}}
\newcommand{\N}{\mathbb{N}}

\newcommand{\ie}{\textit{i}.\textit{e}. }

\DeclareMathOperator{\row}{row}

\begin{document}
\title[Cohomology of type $B$ real permutohedral varieties]{Cohomology of type $B$ real permutohedral varieties}

\author[Younghan Yoon]{Younghan Yoon}
\address{Department of mathematics, Ajou University, 206, World cup-ro, Yeongtong-gu, Suwon 16499,  Republic of Korea}
\email{younghan300@ajou.ac.kr}

\date{\today}
\subjclass[2020]{57S12, 14P25, 14M25, 55U10, 20F55}

\keywords{cohomology rings, real algebraic varieties, type~$B$ permutohedral varieties, signed permutations, $B$-snakes, Coxeter groups, root systems, Weyl groups}

\thanks{This work was supported by the National Research Foundation of Korea Grant funded by
the Korean Government (NRF-2021R1A6A1A10044950).}

\begin{abstract}
Type~$A$ and type~$B$ permutohedral varieties are classic examples of mathematics, and their topological invariants are well known.
This naturally leads to the investigation of the topology of their real loci, known as type~$A$ and type~$B$ real permutohedral varieties.

The rational cohomology rings of type~$A$ real permutohedral varieties are fully described in terms of alternating permutations.
Until now, only rational Betti numbers of type~$B$ real permutohedral varieties have been described in terms of $B$-snakes.
In this paper, we explicitly describe the multiplicative structure of the cohomology rings of type~$B$ real permutohedral varieties in terms of $B$-snakes.
\end{abstract}
\maketitle
\tableofcontents

\section{Introduction}

A normal algebraic variety $X$ is called a \emph{toric variety} if $X$ contains an algebraic torus as an open dense subset such that the action of the torus on itself extends to $X$.
The fundamental theorem for toric geometry says that the category of $n$-dimensional toric varieties is equivalent to that of fans in $\R^n$.
Furthermore, it is well-known that each smooth compact toric variety corresponds to a non-singular complete fan~\cite{Cox_toric_book}.

A root system~$\Phi_R$ of type~$R$ induces a non-singular complete fan~$\Sigma_R$ via its Weyl chambers and co-weight lattice~\cite{Procesi1990}.
Consequently, $\Phi_R$ determines a smooth compact toric variety~$X_R$, known as the \emph{Coxeter toric variety} of type $R$.
In particular, $X_{A_n}$ and $X_{B_n}$ are called the $n$-dimensional \emph{type~$A$} and \emph{type~$B$ permutohedral varieties}, respectively.
They have been extensively studied in various fields of mathematics~\cite{Klyachko1985, Mari1988, JuneHuh2012, Huh2014PhD, Eur2024, Clader2024, Eur2025}.

For a toric variety $X$, the fixed point set $X^\R$ of $X$ under the canonical involution induced by a complex conjugation is called a \emph{real toric variety}.
For a root system of type~$R$, the real locus~$X^\R_R$ is called the \emph{real Coxeter toric variety} of type~$R$.
For types $R=A$ or $R = B$, $X^\R_R$ is called a \emph{type~$A$} or \emph{type~$B$ real permutohedral variety}, respectively.

In general, computing topological invariants of real algebraic varieties is quite challenging.
However, due to the effective symmetry of real Coxeter toric varieties, the rational Betti numbers of $X^{\R}_R$ have been fully computed for all root systems~\cite{Henderson2012, Choi-Park-Park2017typeB, Choi-Kaji-Park2019, Cho-Choi-Kaji2019, Choi-Yoon-Yu2024}.
It is noteworthy that, for the types $R = A$ and $B$, the $k$th (rational) Betti numbers $\beta_k$ of $X^{\R}_{R}$ are given by:
\begin{align*}
  \beta^k(X^{\R}_{A_n}) &= \binom{n+1}{2k} a_{2k} \\
  \beta_k(X^{\R}_{B_n}) &= {n \choose 2k}b_{2k} + {n \choose 2k-1}b_{2k-1},
\end{align*}
where $a_r$ denotes the $r$th \emph{Euler zigzag number} (A000111 in \cite{oeis}) and $b_r$ denotes the $r$th \emph{Springer number} (A001586 in \cite{oeis}).
The numbers $a_r$ and $b_r$ are also known as the numbers of alternating permutations and $B_r$-snakes, respectively.

Very little has been studied about the multiplicative structure of the rational cohomology $H^\ast(X^\R_{R}; \Q)$ of a real Coxeter toric variety $X^\R_R$.
However, it is known that the multiplicative structure of $H^\ast(X^\R_{A_n}; \Q)$ can be described in terms of alternating permutations \cite{Choi-Yoon2023}.
In this paper, we describe the multiplicative structure of $H^\ast(X^\R_{B_n}; \Q)$ in terms of $B$-snakes.

For each subset $I$ of $[n] = \{1,\ldots,n\}$, a \emph{signed permutation} on $I$ is a map $x \colon I \to \Z$ that satisfies the following conditions:
\begin{enumerate}
  \item $\left\vert x(i) \right\vert \in I$ for each $i \in I$, and
  \item $\left\vert x(i) \right\vert \neq \left\vert x(j) \right\vert$ for $i \neq j$.
\end{enumerate}
The set of signed permutations on $I$ is denoted by $\fS^B_I$, which is also known as the \emph{hyperoctahedral group} and the Weyl group of type~$B$~\cite{Gordon_book1981}.
A signed permutation $x = x_r \cdots x_2 x_1 \in \fS^B_I$ can be grouped as follows:
$$
\begin{cases}
	[x_{2k-1}/x_{2k-2}x_{2k-3}/\cdots/x_2x_1], & r = 2k-1, \\
	[x_{2k}x_{2k-1}/x_{2k-2}x_{2k-3}/\cdots/x_2x_1], & r = 2k.
\end{cases}
$$
A signed permutation $x = x_r\cdots x_2x_1 \in \fS^B_I$ is called a \emph{$B$-snake} on $I$, if
$$
0 < x_r > x_{r-1} < \cdots x_1,
$$
and the set of $B$-snakes on $I$ is denoted by $\fA^B_I$.

We consider the $\Q$-vector spaces $\Q \left\langle \fS_I^B \right\rangle$ and $\Q \left\langle \fA_I^B \right\rangle$.
Each element $x \in \fS^B_I$ and $y \in \fA^B_I$ corresponds to a vector in $\Q \left\langle \fS_I^B \right\rangle$ and $\Q \left\langle \fA_I^B \right\rangle$, respectively.
\begin{definition}\label{def M_I}
For each $x = x_r\cdots x_1 \in \fS^B_I$, we define some elements of $\Q \left\langle \fS_I^B \right\rangle$ as follows.
\begin{enumerate}
  \item For each $1 \leq i  \leq \lfloor \frac{r}{2}\rfloor$, $\cH_1^i(x) := [x_r \cdots /x_{2i}x_{2i-1}/\cdots x_1] + [x_r \cdots /x_{2i-1}x_{2i}/\cdots x_1]$.
  \item For $1 \leq i < \lfloor \frac{r}{2}\rfloor$, $\cH_2^i(x)$ is defined by
		\begin{align*}
		[x_{r} \cdots /x_{2i+2} x_{2i+1} /x_{2i} x_{2i-1}/ \cdots x_1] & - [x_{r} \cdots /x_{2i+2} x_{2i}/ x_{2i+1} x_{2i-1}/ \cdots x_1] \\
			+ [x_{r} \cdots /x_{2i+1} x_{2i}/ x_{2i+2} x_{2i-1}/ \cdots x_1] & + [x_{r} \cdots /x_{2i+2} x_{2i-1}/ x_{2i+1} x_{2i}/ \cdots x_1] \\
			- [x_{r} \cdots /x_{2i+1} x_{2i-1} /x_{2i+2} x_{2i}/ \cdots x_1] & + [x_{r} \cdots /x_{2i} x_{2i-1}/ x_{2i+2} x_{2i+1}/ \cdots x_1].
		\end{align*}
    \item If $\left\vert I \right\vert$ is odd, $\cH_3(x) := [x_r/x_{r-1}\cdots x_1] + [\bar{x}_{r}/ x_{r-1}\cdots x_1]$.
        Otherwise, $\cH_3(x):=0$.
    \item If $\left\vert I \right\vert$ is even, $$
        \cH_4(x) := [x_rx_{r-1}/ \cdots x_1] - [x_r\bar{x}_{r-1}/\cdots x_1] + [x_{r-1}\bar{x}_{r} /\cdots x_1] -[\bar{x}_{r-1}\bar{x}_r /\cdots x_1].
        $$
        Otherwise, $\cH_4(x):=0$.
    \item If $\left\vert I \right\vert$ is odd and at least $3$, $\cH_5(x)$ is defined by
		\begin{align*}
			[x_r/x_{r-1}x_{r-2}/\cdots ] &- [x_r/\bar{x}_{r-1}x_{r-2}/\cdots ] + [x_r/\bar{x}_{r-2}x_{r-1}/\cdots ] - [x_r/\bar{x}_{r-2}\bar{x}_{r-1}/\cdots]\\
			-[x_{r-1}/x_{r}x_{r-2}/\cdots ] &+ [x_{r-1}/\bar{x}_rx_{r-2}/\cdots ] - [x_{r-1}/\bar{x}_{r-2}x_r/\cdots ] + [x_{r-1}/\bar{x}_{r-2}\bar{x}_r/\cdots ]\\
			+[x_{r-2}/x_rx_{r-1}/\cdots ] &- [x_{r-2}/\bar{x}_{r}x_{r-1}/\cdots ] + [x_{r-2}/\bar{x}_{r-1}x_{r}/\cdots ] - [x_{r-2}/\bar{x}_{r-1}\bar{x}_{r}/\cdots ].
		\end{align*}
    Otherwise, $\cH_5(x):=0$.
\end{enumerate}
For $1 \leq i \leq 5$, let $\cH_i$ be a set defined by
$$
\begin{cases}
  \{\cH_1^j(x) \colon x \in \fS^B_I, 1 \leq j  \leq \lfloor \frac{r}{2}\rfloor\}, & \mbox{if } i =1 \\
  \{\cH_2^j(x) \colon x \in \fS^B_I, 1 \leq j  < \lfloor \frac{r}{2}\rfloor \}, & \mbox{if } i =2 \\
  \{\cH_i(x) \colon x \in \fS^B_I\}, & \mbox{otherwise}.
\end{cases}
$$
Let $M_I$ be a $\Q$-subspace of $\Q \left\langle \fS_I^B \right\rangle$ spanned by $\cH_i$ for all $1 \leq i \leq 5$.
\end{definition}

In Section~\ref{Sec:signed perm}, we define the explicit $\Q$-linear map
$$
\sum_{I \subset [n]} \Q\left\langle \fS^B_I \right\rangle \to H^\ast(X^\R_{B_n}; \Q),
$$
and demonstrate that the kernel of this map includes $M_I$ for all $I \subset [n]$.
In conclusion, this induces a well-defined $\Q$-linear map
$$
\Psi \colon \sum_{I \subset [n]} \Q\left\langle \fS^B_I \right\rangle/{M_I} \to H^\ast(X^\R_{B_n}; \Q).
$$

In Section~\ref{Sec:B-snakes}, we prove Theorem~\ref{theorem_in4} using the newly defined partial order $\trianglelefteq$ on $\fS^B_I$.
Furthermore, this establishes Theorem~\ref{main1}.

\begin{theorem}\label{main1}
For a subset $I \subset [n]$ with cardinality $r$,
\begin{enumerate}
  \item the quotient space $\Q\left\langle \fS^B_I \right\rangle / M_I$ is isomorphic to $\Q\left\langle \fA^B_I \right\rangle$, and
  \item the $\Q$-linear map $\Psi$ is a $\Q$-vector space isomorphism.
\end{enumerate}
\end{theorem}
Every element $x \in \fS^B_I$ can be written as a linear combination of $B$-snakes on $I$ within $\Q\left\langle \fS^B_I  \right\rangle/ M_I$, and $\cC^{\alpha}_x$ denotes the coefficient of $\alpha \in \fA^B_I$ in the expression of~$x$, as
$$
	x = \sum_{\alpha \in \fA^B_I}\cC^{\alpha}_x\cdot\alpha.
$$

For subsets $I_1$ and $I_2$ of $[n]$, we briefly introduce some notions:
\begin{itemize}
	\item When $I_1 \cap I_2 = \emptyset$ and $\left\vert I_1 \right\vert \cdot \left\vert I_2 \right\vert$ is even, a $B$-snake $z_{\ell}\cdots z_1 \in \fA^B_{I_1 \cup I_2}$ is said to be \emph{restrictable} to $(I_1,I_2)$, if $\{\left\vert z_{2i-1} \right\vert , \left\vert z_{2i} \right\vert \}$ is contained in either $I_1$ or $I_2$ for each $1 \leq i \leq \lfloor \frac{\ell-1}{2} \rfloor$.
	\item $\cR^B_{(I_1,I_2)}$ is the set of restrictable $B$-snakes to $(I_1,I_2)$.
	\item For each $z = z_\ell \cdots z_1 \in \cR^B_{(I_1,I_2)}$, $\kappa_{(I_1,I_2)}(z)$ is the number of pairs $(z_{2i-1},z_{2j-1})$ such that $\left\vert z_{2i-1} \right\vert \in I_1$ and $\left\vert z_{2j-1} \right\vert \in I_2$ for all $\lfloor \frac{\ell+1}{2} \rfloor \geq i > j \geq 1$.
	\item For each $x \in \fS^B_I$, $\bar{x}$ is a signed permutation on $I$ such that $\bar{x}(i) = -x(i)$ for each $i \in I$.
	\item For $J \subset I \subset [n]$, $\rho_J(x) \in \fS^B_J$ is a \emph{subpermutation} of $x$, and $P_J(x)$ is as follows:
	$$
    P_J(x) =
	\begin{cases}
		\rho_{J}(x), & \mbox{if $\left\vert I \right\vert + \left\vert J \right\vert$ is even}, \\
		\rho_{J}(\bar{x}), & \mbox{if $\left\vert I \right\vert + \left\vert J \right\vert$ is odd},
	\end{cases}
	$$
\end{itemize}

\begin{definition}\label{def-multi}
	Let $\smile \colon \Q\left\langle \fA^B_{I_1}  \right\rangle \otimes \Q\left\langle \fA^B_{I_2}  \right\rangle \to \Q\left\langle \fA^B_{I_1 \triangle I_2}  \right\rangle$ be defined by, for each~$\alpha \in \fA^B_{I_1}$ and~$\beta \in \fA^B_{I_2}$,
	$$
	\alpha \smile \beta = \begin{cases}
		\underset{z \in \cR_{(I_1,I_2)}}{\sum}(-1)^{\kappa_{(I_1,I_2)}(z)} \cC^\alpha_{P_{I_1}(z)} \cC^\beta_{P_{I_2}(z)} \cdot z, & \mbox{if } I_1 \cap I_2 = \emptyset, \left\vert I_1 \right\vert \cdot \left\vert I_2 \right\vert \text{ is even},\\
		0, & \mbox{otherwise,}
	\end{cases}
	$$
	where $I_1 \triangle I_2 = (I_1 \cup I_2) \setminus (I_1 \cap I_2)$.
\end{definition}

In Section~\ref{Sec:multi}, we prove Theorem~\ref{thm:final}, which confirms Theorem~\ref{main2}.
\begin{theorem}\label{main2}
 	 From the multiplicative structure $\smile$,
 	 $$
 	( \bigoplus_{I \subset [n]} \Q\left\langle \fA^B_{I}  \right\rangle,+,\smile)
 	 $$
 	 is isomorphic to $\underset{k \in \N}\bigoplus H^k(X_{B_n}^{\R},\Q)$ as a $\Q$-algebra.
\end{theorem}

\section{Preliminaries}

Let $\Phi_R$ be a root system of type~$R$ in a finite-dimensional Euclidean space~$E$, and $W_R$ its Weyl group.
The group $W_R$ partitions $E$ into connected components, known as the \emph{Weyl chambers} of type~$R$.
Consider the fan~$\Sigma_R$ whose maximal cones are Weyl chambers of type~$R$.
It is worth noting that $\Sigma_R$ is a non-singular complete fan~\cite{Procesi1990}.
Let $V_R = \{v_1,v_2,\ldots,v_n\}$ denote the set of rays generating the fan~$\Sigma_R$, and $K_R$ the simplicial complex, known as the \emph{Coxeter complex} of type~$R$, whose faces correspond to the corresponding faces in $\Sigma_R$.
We define a linear map $\lambda_R \colon V_R \to \Z^n$, called the \emph{characteristic map} of type~$R$, such that $\lambda_R(v)$ is the primitive vector in the direction of each $v \in V_R$.
The fan $\Sigma_R$ can be replaced by the pair $(K_R,\lambda_R)$.
By the fundamental theorem of toric geometry, the pair $(K_R,\lambda_R)$ completely determines a smooth compact toric variety~$X_R$, referred to as the \emph{Coxeter toric variety} of type~$R$.

A \emph{real Coxeter toric variety} $X^\R_R$ of type~$R$, that is, the real locus of $X_R$, is also completely determined by a pair $(K_R,\lambda^\R_R)$, where the (mod $2$) characteristic map is defined as
$$
\lambda^{\R}_R \colon V_R \xrightarrow{\lambda_R} \Z^n \xrightarrow{\text{mod} 2} (\Z/2\Z)^n.
$$
The map $\lambda_R^\R$ can also be regarded as an $n \times m$ (mod $2$) matrix $\Lambda_R$, called the (mod $2$) \emph{characteristic matrix} of $X^\R_R$.

The row space $\row(\Lambda_R)$ of $\Lambda_R$ forms a subspace of $(\Z/2\Z)^m$.
Each element $S \in \row(\Lambda_R)$ naturally corresponds to a subset of $V_R$.
The cohomology of $X^{\R}_R$ is known to depend entirely on the pair $(K_R, \Lambda_R)$, as established in Theorem~\ref{ChoiPark}.
For further details, see \cite{Choi-Park2017_multiplicative}.
\begin{theorem}\label{ChoiPark}
  There is a $\Q$-algebra isomorphism
$$
  H^{\ast}(X^{\R}(K_R),\Q) \cong \underset{S \in \row(\Lambda_R)}\bigoplus \widetilde{H}^{\ast-1}((K_R)_S,\Q).
$$
The multiplicative structure on $\underset{S \in \row(\Lambda_R)}\bigoplus \widetilde{H}^{\ast}((K_R)_S;\Q)$ is defined via the canonical maps
$$
  \widetilde{H}^{k-1}((K_R)_{S_1}) \otimes \widetilde{H}^{\ell-1}((K_R)_{S_2}) \to \widetilde{H}^{k+\ell-1}((K_R)_{S_1+S_2})
$$
induced by simplicial maps $(K_R)_{S_1+S_2} \to (K_R)_{S_1} \star (K_R)_{S_2}$, where $\star$ denotes the simplicial join.
\end{theorem}

We now recall the constructions of $K_{B_n}$ and $\Lambda_{B_n}$.
As described in~\cite{Choi-Park-Park2017typeB}, each $k$-simplex of $K_{B_n}$ can be identified with a nested chain of $k+1$ nonempty proper subsets $L_1 \subsetneq \cdots \subsetneq L_{k+1}$ of $[n] = \{1,\ldots,n\}$, satisfying the condition:
$$
\text{if } i \in L_s, \text{ then } -i \notin L_s
$$
for all $1 \leq s \leq k+1$.
The (mod $2$) characteristic map $\lambda^{\R}_{B_n}$ is defined by 
$$
\lambda^{\R}_{B_n}(L) = \sum_{k \in L}\mathbf{e}_k + \sum_{-k \in L}\mathbf{e}_{-k}
$$
for vertex $L$ in $K_{B_n}$, where $\mathbf{e}_r$ is the $r$th standard vector of $(\Z/2\Z)^n$ for $1 \leq r \leq n$.

Each element $S$ in $\row(\Lambda_{B_n})$ corresponds to a subset $I_S$ of $[n]$.
For $S \in \row(\Lambda_{B_n})$, let  $(K_{B_n})_{I_S}$ denote the full-subcomplex of $K_{B_n}$ consisting of vertices $J$ such that the cardinality of
$$
\{i \in I_S \colon i \in J\text{ or } -i \in J\}
$$
is odd.
Additionally, for $S_1$ and $S_2 \in \row(\Lambda_{B_n})$, we have
$$
I_{S_1+S_2} = (I_{S_1} \cup I_{S_2})\setminus(I_{S_1} \cap I_{S_2}).
$$
Therefore, each subset $I \subset [n]$ can be regarded as an element of~$\Lambda_{B_n}$.

Let us define the subcomplex $\widehat{(K_{B_{n}})_{I}}$ of $(K_{B_{n}})_{I}$ obtained from $(K_{B_{n}})_{I}$ by iteratively removing all vertices $L \not\subset I \cup \{-i \colon i \in I\}$.

\begin{proposition}\cite{Choi-Park-Park2017typeB}\label{typeB_homotopy}
	Let $\pi \colon (K_{B_{n}})_{I} \to \widehat{(K_{B_{n}})_{I}}$ be the simplicial map defined by $\pi(\sigma) = \sigma \cap I$ for each $\sigma \in (K_{B_{n}})_{I}$.
	Then $\pi$ is a homotopy equivalence.
\end{proposition}

Consequently, when computing (co)homology of $(K_{B_n})_I$, it suffices to consider only the full-subcomplex $\widehat{(K_{B_{n}})_{I}}$, as any contributions from vertices not contained within this subcomplex are irrelevant.

\begin{example}
Consider $B_3$ and $I = \{1\}$.
Then, $(K_{B_3})_{I}$ consists of two copies of the $2$-dimensional component as follows:
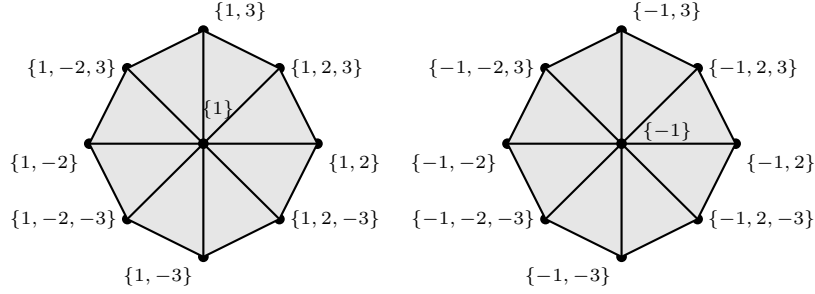
\begin{center}
\begin{tikzpicture}[scale = 1]
            \coordinate (1) at (0,0);
            \coordinate (1') at (0.19,0.2);
            \coordinate (12) at (1.5,0); \fill (12) circle [radius=2pt]; \node [below right] at (12) {\tiny $\{1,2\}$};
            \coordinate (13) at (0,1.5); \fill (13) circle [radius=2pt]; \node [above right] at (13) {\tiny $\{1,3\}$};
            \coordinate (1b2) at (-1.5,0); \fill (1b2) circle [radius=2pt]; \node [below left] at (1b2) {\tiny $\{1,-2\}$};
            \coordinate (1b3) at (0,-1.5); \fill (1b3) circle [radius=2pt]; \node [below left] at (1b3) {\tiny $\{1,-3\}$};
            \coordinate (123) at (1,1); \fill (123) circle [radius=2pt]; \node [right] at (123) {\tiny $\{1,2,3\}$};
            \coordinate (1b23) at (-1,1); \fill (1b23) circle [radius=2pt]; \node [left] at (1b23) {\tiny $\{1,-2,3\}$};
            \coordinate (12b3) at (1,-1); \fill (12b3) circle [radius=2pt]; \node [right] at (12b3) {\tiny $\{1,2,-3\}$};\coordinate (1b2b3) at (-1,-1); \fill (1b2b3) circle [radius=2pt]; \node [left] at (1b2b3) {\tiny $\{1,-2,-3\}$};
    \fill[gray!20] (1)--(12)--(123)--cycle;
    \fill[gray!20] (1)--(123)--(13)--cycle;
    \fill[gray!20] (1)--(13)--(1b23)--cycle;
    \fill[gray!20] (1)--(1b23)--(1b2)--cycle;
    \fill[gray!20] (1)--(1b2)--(1b2b3)--cycle;
    \fill[gray!20] (1)--(1b2b3)--(1b3)--cycle;
    \fill[gray!20] (1)--(1b3)--(12b3)--cycle;
    \fill[gray!20] (1)--(12b3)--(12)--cycle;
    
    \foreach \point in {12, 123, 13, 1b23, 1b2, 1b2b3, 1b3, 12b3} {
    \draw [thick, miter limit=2] (1)--(\point);}
    
    \draw [thick, miter limit=2] 
        (12)--(123)--(13)--(1b23)--(1b2)--(1b2b3)--(1b3)--(12b3)--cycle;
   \fill (1) circle [radius=2pt];     
   \node [above] at (1') {\tiny $\{1\}$};  
   
    \coordinate (b1) at (5.5,0);
    \coordinate (b1') at (6.1,-0.1);
    \coordinate (b12) at (7,0); \fill (b12) circle [radius=2pt]; \node [below right] at (b12) {\tiny $\{-1,2\}$};
    \coordinate (b13) at (5.5,1.5); \fill (b13) circle [radius=2pt]; \node [above right] at (b13) {\tiny $\{-1,3\}$};
    \coordinate (b1b2) at (4,0); \fill (b1b2) circle [radius=2pt]; \node [below left] at (b1b2) {\tiny $\{-1,-2\}$};
    \coordinate (b1b3) at (5.5,-1.5); \fill (b1b3) circle [radius=2pt]; \node [below left] at (b1b3) {\tiny $\{-1,-3\}$};
    \coordinate (b123) at (6.5,1); \fill (b123) circle [radius=2pt]; \node [right] at (b123) {\tiny $\{-1,2,3\}$};
    \coordinate (b1b23) at (4.5,1); \fill (b1b23) circle [radius=2pt]; \node [left] at (b1b23) {\tiny $\{-1,-2,3\}$};
    \coordinate (b12b3) at (6.5,-1); \fill (b12b3) circle [radius=2pt]; \node [right] at (b12b3) {\tiny $\{-1,2,-3\}$};
    \coordinate (b1b2b3) at (4.5,-1); \fill (b1b2b3) circle [radius=2pt]; \node [left] at (b1b2b3) {\tiny $\{-1,-2,-3\}$};
    \fill[gray!20] (b1)--(b12)--(b123)--cycle;
    \fill[gray!20] (b1)--(b123)--(b13)--cycle;
    \fill[gray!20] (b1)--(b13)--(b1b23)--cycle;
    \fill[gray!20] (b1)--(b1b23)--(b1b2)--cycle;
    \fill[gray!20] (b1)--(b1b2)--(b1b2b3)--cycle;
    \fill[gray!20] (b1)--(b1b2b3)--(b1b3)--cycle;
    \fill[gray!20] (b1)--(b1b3)--(b12b3)--cycle;
    \fill[gray!20] (b1)--(b12b3)--(b12)--cycle;
    \foreach \point in {b12, b123, b13, b1b23, b1b2, b1b2b3, b1b3, b12b3} {
        \draw [thick, miter limit=2] (b1)--(\point);
    }
    \draw [thick, miter limit=2] 
        (b12)--(b123)--(b13)--(b1b23)--(b1b2)--(b1b2b3)--(b1b3)--(b12b3)--cycle;
    \fill (b1) circle [radius=2pt];
    \node [above] at (b1') {\tiny $\{-1\}$}; 
        \end{tikzpicture}
     \end{center}
Moreover, $\widehat{(K_{B_3})}_{I}$ consists of two $0$-simplices labeled as $\{1\}$ and $\{-1\}$.
The simplicial map defined in Proposition~\ref{typeB_homotopy} indeed forms a homotopy equivalence.
\end{example}

From~\cite{Choi-Park-Park2017typeB}, $\widehat{(K_{B_{n}})_{I}}$ is a shellable complex, which implies that it is Cohen-Macaulay~\cite{Stanley1996book}.
Thus, $\widehat{(K_{B_{n}})_{I}}$ is homotopy equivalent to a bouquet of spheres of the same dimension.
From the Euler characteristic, the (rational) Betti numbers $\beta$ can be computed as follows:
\begin{equation}\label{typeB_I}
	\tilde{\beta}_{\lfloor \frac{r-1}{2} \rfloor}((K_{B_n})_I) = b_{k},
\end{equation}
where $I \subset [n]$ and $\left\vert I \right\vert = r$.

\begin{remark}\cite{Choi-Park-Park2017typeB}
  The integral cohomology of $X^\R_{B_n}$ is $p$-torsion free for all odd primes $p$.
\end{remark}

\section{Signed Permutations as elements in $H_\ast(X^\R_{B_n};\Q)$}\label{Sec:signed perm}

In this section, assume that $I$ is a subset of $[n] =\{1,\ldots,n\}$ with cardinality $r$.
A \emph{signed permutation} on $I$ is defined as a map
$$
x \colon I \to I \cup \{-i \colon i \in I\}
$$
such that $\left\vert x(i) \right\vert \neq \left\vert x(j)\right\vert$ whenever $i \neq j$, and $\fS^B_I$ denotes the set of signed permutations on $I$.
A signed permutation $x$ on $I$ is represented as  $x_r\cdots x_2x_1$, where 
$$
x^{-1}(x_r) < \cdots < x^{-1}(x_{2}) < x^{-1}(x_1).
$$
This representation can alternatively be grouped as:
$$
\begin{cases}
	[x_{2k-1}/x_{2k-2}x_{2k-3}/\cdots/x_2x_1], & r = 2k-1, \\
	[x_{2k}x_{2k-1}/x_{2k-2}x_{2k-3}/\cdots/x_2x_1], & r = 2k,
\end{cases}
$$
where terms are divided into blocks of length two, starting from the last pair.

For $x= x_r\cdots x_2x_1 \in \fS^B_I$, we define some notations:
\begin{itemize}
  \item $\bar{x}_i = -x_i$ for each $1 \leq i \leq r$,
  \item $x^\ast = \begin{cases}
	[\bar{x}_{2k-1}/x_{2k-3}x_{2k-2}/\cdots/ x_{2i-1}x_{2i}/ \cdots /x_{1} x_{2}] , & \mbox{if } r = 2k-1, \\
	[x_{2k-1}x_{2k}/\cdots/ x_{2i-1}x_{2i}/ \cdots /x_{1} x_{2}], & \mbox{if } r = 2k,
\end{cases}$
and
\item for each $1 \leq i \leq \lfloor \frac{r+1}{2} \rfloor$,
$$
\cF_i(x) = \begin{cases}
             \{x_1,\ldots,x_{2i-1}\}, & \mbox{if $r$ is odd}, \\
             \{\bar{x}_1,\ldots,\bar{x}_{2i-1}\}, & \mbox{if $r$ is even}.
           \end{cases}
$$
\end{itemize}

\begin{example}
    For $I = \{1,2,3,4,5\}$ and $J = \{1,2,3,4\}$, let us consider signed permutations $x = [1/\bar{3}2/\bar{5}\bar{4}] \in \fS_I^B$ and $y = [1\bar{4}/32] \in \fS_J^B$.
    We obtain that
    $$
    \begin{cases}
      \cF_1(x) = \{\bar{4}\}, \cF_2(x) = \{2,\bar{4},\bar{5}\}, \cF_3(x) = \{1,2,\bar{3},\bar{4},\bar{5}\}, \\
      \cF_1(y) = \{\bar{2}\}, \cF_2(y) = \{\bar{2},\bar{3},4\}.
    \end{cases}
    $$
    Note that $x^\ast = [\bar{1}/2\bar{3}/\bar{4}\bar{5}], y^\ast = [\bar{4}1/23]$.
    Then,
    $$
    \begin{cases}
      \cF_1(x^\ast) = \{\bar{5}\}, \cF_2(x^\ast) = \{\bar{3},\bar{4},\bar{5}\}, \cF_3(x^\ast) = \{\bar{1},2,\bar{3},\bar{4},\bar{5}\}, \\
      \cF_1(y^\ast) = \{\bar{3}\}, \cF_2(y^\ast) = \{\bar{1},\bar{2},\bar{3}\}.
    \end{cases}
    $$
\end{example}

We associate to $x$ a subcomplex $\square_x$ of $(K_{B_n})_I$ as
$$
\square_x = \{ \cF_1(x), \cF_1(x^\ast)\} \star \{\cF_2(x), \cF_2(x^\ast)\} \star \cdots \star \{\cF_{\lfloor \frac{r+1}{2} \rfloor}(x), \cF_{\lfloor \frac{r+1}{2} \rfloor}(x^\ast)\}
$$
where $\star$ is the simplicial join.
It is noteworthy that $\square_x$ is the boundary complex of the $k$-dimensional cross-polytope with fixed orientation embedded in $(K_{B_n})_I$.

Let $\Psi_I \colon \Q \left\langle \fS_I^B \right\rangle \to \widetilde{H}_{\lfloor \frac{r-1}{2} \rfloor}((K_{B_n})_I)$ be a $\Q$-linear map defined by
$$
\Psi_I(x) = [\square_x]
$$
for each $x \in \fS_I^B$.

\begin{example}
	Consider finite sets $I = \{1,3,4\}$ and $J = \{1,2,3,4\}$.
	Let $x = [1/\bar{4} 3] \in \fS_I^B $ and $y = [1\bar{4}/32] \in \fS_J^B $.
	We have $x^\ast = [\bar{1}/3\bar{4}]$, $y^\ast = [\bar{4}1/23]$.
	Then, we conclude that
	\begin{align*}
		\Psi_{I}(x) & = [\{\cF_1(x),\cF_1(x^\ast)\}\star\{\cF_2(x),\cF_2(x^\ast)\}] \\
        &= [\{\{3\},\{\bar{4}\}\}\star \{\{3,\bar{4},1\},\{3,\bar{4},\bar{1}\}\}] \\
		&= [\{\{3\},\{3,\bar{4},1\}\}- \{\{3\},\{3,\bar{4},\bar{1}\}\}-\{\{\bar{4}\},\{3,\bar{4},1\}\}+ \{\{\bar{4}\},\{3,\bar{4},\bar{1}\}\}],\\
		\Psi_{J}(y) & = [\{\cF_1(y),\cF_1(y^\ast)\}\star\{\cF_2(y),\cF_2(y^\ast)\}] \\ &= [\{\{\bar{2}\},\{\bar{3}\}\}\star\{\{\bar{2},\bar{3},4\},\{\bar{2},\bar{3},\bar{1}\}\}]\\ &= [\{\{\bar{2}\},\{\bar{2},\bar{3},4\}\}-\{\{\bar{2}\},\{\bar{2},\bar{3},\bar{1}\}\}-\{\{\bar{3}\},\{\bar{2},\bar{3},4\}\}+\{\{\bar{3}\},\{\bar{2},\bar{3},\bar{1}\}\}].
	\end{align*}
\end{example}

\begin{example}\label{example:K_3_I}
    Let us consider the type $B_3$.
    See Figure~\ref{fig:KBI3} for the induced full-subcomplex~$(K_{B_3})_{I}$ of~$K_{B_3}$, where~$I = \{1,2,3\}$.
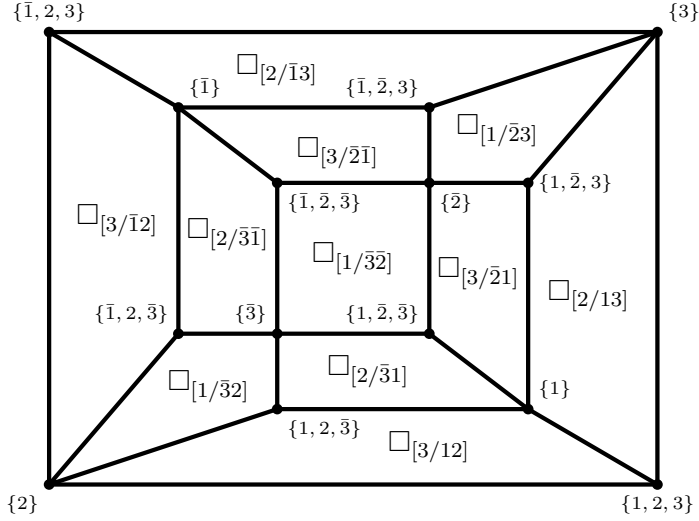
\begin{figure}
\begin{center}
    \begin{tikzpicture}[scale = 1]
    \coordinate (b1b2b3) at (-1,1); \fill (b1b2b3) circle [radius=2pt]; \node [below right] at (b1b2b3) {\tiny $\{\bar{1},\bar{2},\bar{3}\}$};
    \coordinate (1b2b3) at (1,-1); \fill (1b2b3) circle [radius=2pt]; \node [above left] at (1b2b3) {\tiny $\{1,\bar{2},\bar{3}\}$};
    \coordinate (b1b23) at (1,2); \fill (b1b23) circle [radius=2pt]; \node [above left] at (b1b23) {\tiny $\{\bar{1},\bar{2},3\}$};
    \coordinate (1b23) at (2.3,1); \fill (1b23) circle [radius=2pt]; \node [right] at (1b23) {\tiny $\{1,\bar{2},3\}$};
    \coordinate (1) at (2.3,-2); \fill (1) circle [radius=2pt]; \node [above right] at (1) {\tiny $\{1\}$};
    \coordinate (3) at (4,3); \fill (3) circle [radius=2pt]; \node [above right] at (3) {\tiny $\{3\}$};
    \coordinate (2) at (-4,-3); \fill (2) circle [radius=2pt]; \node [below left] at (2) {\tiny $\{2\}$};
    \coordinate (b1) at (-2.3,2); \fill (b1) circle [radius=2pt]; \node [above right] at (b1) {\tiny $\{\bar{1}\}$};
    \coordinate (b2) at (1,1); \fill (b2) circle [radius=2pt]; \node [below right] at (b2) {\tiny $\{\bar{2}\}$};
    \coordinate (b3) at (-1,-1); \fill (b3) circle [radius=2pt]; \node [above left] at (b3) {\tiny $\{\bar{3}\}$};
    \coordinate (12b3) at (-1,-2); \fill (12b3) circle [radius=2pt]; \node [below right ] at (12b3) {\tiny $\{1,2,\bar{3}\}$};
    \coordinate (b12b3) at (-2.3,-1); \fill (b12b3) circle [radius=2pt]; \node [above left] at (b12b3) {\tiny $\{\bar{1},2,\bar{3}\}$};
    \coordinate (123) at (4,-3); \fill (123) circle [radius=2pt]; \node [below] at (123) {\tiny $\{1,2,3\}$};
    \coordinate (b123) at (-4,3); \fill (b123) circle [radius=2pt]; \node [above] at (b123) {\tiny $\{\bar{1},2,3\}$};
    
    \draw [ultra thick, miter limit=2] (123)--(2)--(b123)--(3)--cycle;
    \draw [ultra thick, miter limit=2] (1b2b3)--(b3)--(b1b2b3)--(b2)--cycle;
    \draw [ultra thick, miter limit=2] (12b3)--(2)--(b12b3)--(b3)--cycle;
    \draw [ultra thick, miter limit=2] (b2)--(1b23)--(3)--(b1b23)--cycle;
    \draw [ultra thick, miter limit=2] (1)--(12b3);
    \draw [ultra thick, miter limit=2] (1)--(123);
    \draw [ultra thick, miter limit=2] (1)--(1b2b3);
    \draw [ultra thick, miter limit=2] (1)--(1b23);
    \draw [ultra thick, miter limit=2] (b1)--(b1b23);
    \draw [ultra thick, miter limit=2] (b1)--(b123);
    \draw [ultra thick, miter limit=2] (b1)--(b12b3);
    \draw [ultra thick, miter limit=2] (b1)--(b1b2b3);
    
    \node at (0,0) {$\square_{[1/\bar{3}\bar{2}]}$};
    \node at (-0.2,1.4) {$\square_{[3/\bar{2}\bar{1}]}$};
    \node at (0.2,-1.5) {$\square_{[2/\bar{3}1]}$};
    \node at (1.65,-0.2) {$\square_{[3/\bar{2}1]}$};
    \node at (-1.65,0.32) {$\square_{[2/\bar{3}\bar{1}]}$};
    \node at (1.9,1.7) {$\square_{[1/\bar{2}3]}$};
    \node at (-1.9,-1.7) {$\square_{[1/\bar{3}2]}$};
    \node at (3.1,-0.5) {$\square_{[2/13]}$};
    \node at (-3.1,0.5) {$\square_{[3/\bar{1}2]}$};
    \node at (-1,2.5) {$\square_{[2/\bar{1}3]}$};
    \node at (1,-2.5) {$\square_{[3/12]}$};
\end{tikzpicture}
     \end{center}
     \caption{A full-subcomplex $(K_{B_{3}})_I$ of $K_{B_{3}}$}\label{fig:KBI3}
     \end{figure}
Note that the outer square-shaped complex is $\square_{[1/23]}$ in this figure, and it follows that 
\begin{align*}
  \Psi_I([1/23]) = & \Psi_I([1/\bar{2}3]) - \Psi_I([1/\bar{3}2]) + \Psi_I([1/\bar{3}\bar{2}]) + \Psi_I([2/13]) - \Psi_I([2/\bar{1}3]) +\Psi_I([2/\bar{3}1]) \\
   & - \Psi_I([2/\bar{3}\bar{1}]) - \Psi_I([3/12]) +\Psi_I([3/\bar{1}2]) -\Psi_I([3/\bar{2}1]) +\Psi_I([3/\bar{2}\bar{1}]).
\end{align*}
\end{example}

\begin{example}\label{ex_K2}
Let $n \geq 4$, and~$I = \{x_{2i-1},x_{2i},x_{2i+1},x_{2i+2}\} \subset [n] = \{1,\ldots,n\}$.
Consider the full-subcomplex~$(K_{B_n})_I$ of the Coxeter complex $K_{B_n}$.
See Figure~\ref{fig:subsub} for a subcomplex of~$(K_{B_n})_I$ induced by
$$
\big\{\{x_1\},\{x_2\},\{x_3\},\{x_4\},\{x_1,x_2,x_3\},\{x_1,x_2,x_4\},\{x_1,x_3,x_4\},\{x_2,x_3,x_4\}\big\}.
$$
The outer square-shaped complex is $\square_{[\bar{x}_4\bar{x}_3/\bar{x}_2\bar{x}_1]}$.
\begin{figure}
  \begin{center}
        \begin{tikzpicture}[scale = 1]
    \coordinate [label = below right:{\tiny $\{x_3\}$}] (3) at (-1,1);
    \coordinate [label = above left:{\tiny $\{x_4\}$}] (4) at (1,-1);
    \coordinate [label = above right:{\tiny$\{x_1\}$}] (1) at (3,2);
    \coordinate [label = below left:{\tiny$\{x_2\}$}] (2) at (-3,-2);
    \coordinate [label = below right:{\tiny $\{x_1,x_3,x_4\}$}] (134) at (1,1);
    \coordinate [label = above left:{\tiny $\{x_2,x_3,x_4\}$}] (234) at (-1,-1);
    \coordinate [label = below:{\tiny $\{x_1,x_2,x_4\}$}] (124) at (3,-2);
    \coordinate [label = above:{\tiny $\{x_1,x_2,x_3\}$}] (123) at (-3,2);
    
    \foreach \point in {1,2,3,4,134,234,124,123} {
        \fill (\point) circle (2pt);
    }
    
    \draw [ultra thick, miter limit=2] (1)--(123)--(2)--(124)--cycle;
    \draw [ultra thick, miter limit=2] (3)--(134)--(4)--(234)--cycle;
    \draw [ultra thick, miter limit=2] (1)--(134);
    \draw [ultra thick, miter limit=2] (2)--(234);
    \draw [ultra thick, miter limit=2] (3)--(123);
    \draw [ultra thick, miter limit=2] (4)--(124); 
    
    \draw (0,0) node {$\square_{[\bar{x}_2\bar{x}_1/\bar{x}_4\bar{x}_3]}$};
    \draw (2,0) node {$\square_{[\bar{x}_3\bar{x}_2/\bar{x}_4\bar{x}_1]}$};
    \draw (-2,0) node {$\square_{[\bar{x}_4\bar{x}_1/\bar{x}_3\bar{x}_2]}$};
    \draw (0,1.5) node {$\square_{[\bar{x}_4\bar{x}_2/\bar{x}_3\bar{x}_1]}$};
    \draw (0,-1.5) node {$\square_{[\bar{x}_3\bar{x}_1/\bar{x}_4\bar{x}_2]}$};
\end{tikzpicture}
     \end{center}
  \caption{A subcomplex of $(K_{B_n})_I$}\label{fig:subsub}
\end{figure}
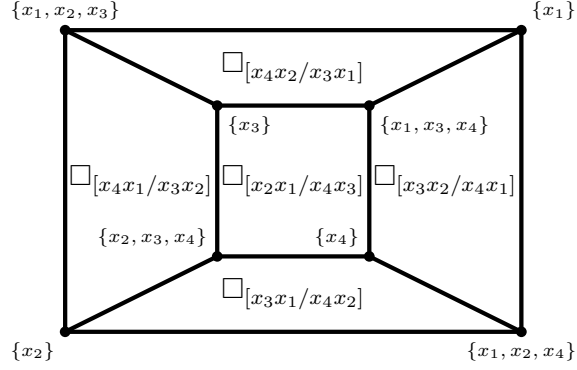
We observe that
$$
  \Psi_I([\bar{x}_4\bar{x}_3/\bar{x}_2\bar{x}_1]) = \Psi_I([\bar{x}_4\bar{x}_2/\bar{x}_3\bar{x}_1] - [\bar{x}_3\bar{x}_2/\bar{x}_4\bar{x}_1] - [\bar{x}_4\bar{x}_1/\bar{x}_3\bar{x}_2] + [\bar{x}_3\bar{x}_1/\bar{x}_4\bar{x}_2] - [\bar{x}_2\bar{x}_1/\bar{x}_4\bar{x}_3]).
$$
\end{example}

Refer to Definition~\ref{def M_I} for a recall of $\cH_i$ and $M_I$.
In the remainder of this section, we demonstrate that the subspace~$M_I$ of $\left\langle \fS_I^B \right\rangle$ is contained in the kernel of~$\Psi_I$.
To begin, we introduce following two lemmas.

\begin{lemma}\label{lemma:sec3_1}
    Let $x = x_r\cdots x_1$ be a signed permutation on $I$. 
    \begin{enumerate}
      \item For each $1 \leq i \leq \lfloor \frac{r}{2} \rfloor$, $\Psi_I(x +[x_r\cdots /x_{2i}x_{2i-1}/\cdots x_{1}]) = 0$.
      \item When $r$ is odd, $\Psi_I(x + [\bar{x}_r/\cdots x_{1}]) = 0$.
    \end{enumerate}
\end{lemma}

\begin{proof}
    For each $1 \leq i \leq \lfloor \frac{r}{2} \rfloor$, 
    \begin{align*}
            \Psi_I(x) & =[\{ \cF_1(x), \cF_1(x^\ast)\} \star \cdots \{\cF_i(x), \cF_i(x^\ast)\} \star \cdots \star \{\cF_{\lfloor \frac{r+1}{2} \rfloor}(x), \cF_{\lfloor \frac{r+1}{2} \rfloor}(x^\ast)\}] \\
             & =-[\{ \cF_1(x), \cF_1(x^\ast)\} \star \cdots \{\cF_i(x^\ast), \cF_i(x)\} \star \cdots \star \{\cF_{\lfloor \frac{r+1}{2} \rfloor}(x), \cF_{\lfloor \frac{r+1}{2} \rfloor}(x^\ast)\}] \\
             & =- \Psi_I([x_r\cdots /x_{2i}x_{2i-1}/\cdots x_{1}]).
          \end{align*}
    Moreover, when $r$ is odd, we have
    \begin{align*}
            \Psi_I(x) & =[\{ \cF_1(x), \cF_1(x^\ast)\} \star \cdots \{\cF_i(x), \cF_i(x^\ast)\} \star \cdots \star \{\cF_{\lfloor \frac{r+1}{2} \rfloor}(x), \cF_{\lfloor \frac{r+1}{2} \rfloor}(x^\ast)\}] \\
             & =-[\{ \cF_1(x), \cF_1(x^\ast)\} \star \cdots \{\cF_i(x), \cF_i(x^\ast)\} \star \cdots \star \{\cF_{\lfloor \frac{r+1}{2} \rfloor}(x^\ast), \cF_{\lfloor \frac{r+1}{2} \rfloor}(x)\}] \\
             & =- \Psi_I([\bar{x}_r/\cdots x_{1}]),
          \end{align*}
          as desired.
\end{proof}

\begin{lemma}\label{lemma:sec3_2}
    Assume that $r$ is even.
    For each $x = [x_rx_{r-1}/\cdots/x_2x_1] \in \fS_I^B$,
    $$
    \Psi_I(x - [x_r\bar{x}_{r-1}/\cdots /x_2x_1] + [x_{r-1}\bar{x}_{r} /\cdots /x_2x_1] -[\bar{x}_{r-1}\bar{x}_r /\cdots /x_2x_1])=0.
    $$
\end{lemma}

\begin{proof}
    For each $r > 2$, define the set $\cG = \{x_1,\ldots,x_{r-2}\}$, and a signed permutation
    $$
    y = [x_{r-2}x_{r-3}/\cdots/x_2x_1] \in \fS^B_{I \setminus \{\left\vert x_r \right\vert, \left\vert x_{r-1} \right\vert\}}
    $$
    For $r =2$, we define both $\cG$ and $\square_y$ to be the empty set.
    Then, we have
    \begin{align*}
      \Psi_I(x) & = [\square_y \star \{\cG\cup\{\bar{x}_{r-1}\}\}] - [\square_y \star \{\cG\cup\{\bar{x}_{r}\}\}], \\
      -\Psi_I([x_r\bar{x}_{r-1}/\cdots /x_2x_1]) & = -[\square_y \star \{\cG\cup\{x_{r-1}\}\}] + [\square_y \star \{\cG\cup\{\bar{x}_{r}\}\}], \\
      \Psi_I([x_{r-1}\bar{x}_{r} /\cdots /x_2x_1]) & = [\square_y \star \{\cG\cup\{x_{r}\}\}] - [\square_y \star \{\cG\cup\{\bar{x}_{r-1}\}\}], \\
      -\Psi_I([\bar{x}_{r-1}\bar{x}_{r} /\cdots /x_2x_1]) & = -[\square_y \star \{\cG\cup\{x_{r}\}\}] + [\square_y \star \{\cG\cup\{x_{r-1}\}\}].
    \end{align*}
    This confirms the lemma.
\end{proof}

By Lemma~\ref{lemma:sec3_1} and Lemma~\ref{lemma:sec3_2}, we have $\cH_1 \cup \cH_3 \subset \ker \Psi_I$ and $\cH_4 \subset \ker \Psi_I$, respectively.
Hence, to confirm that $M_I$ is a subspace of~$\ker \Psi_I$, it remains to verify that~$\ker \Psi_I$ also includes $\cH_2$ and~$\cH_5$.

For a subset $J \subset I$, we define a map $\rho_J \colon \fS^B_I \to \fS^B_J$, where $\rho_J(x)$ is a \emph{subpermutation} of~$x \in \fS_I^B$.
For each $1 \leq i \leq \lfloor \frac{r-1}{2} \rfloor$, define the set $\fI_i$ as
$$
\begin{cases}
  \{r-2,r-1,r\}, & \mbox{if $r$ is odd and $i = \lfloor \frac{r+1}{2} \rfloor -1$},  \\
  \{2i-1,2i,2i+1,2i+2\}, & \mbox{otherwise}.
\end{cases}
$$
Let $x = x_rx_{r-1}\cdots x_1$ be a signed permutation on $I$.
We define a subset $\cT^x_i$ of $\fS_I^B$ by
$$
\cT^x_i = \{y = y_r\cdots y_1 \in \fS_I^B \colon x_j = y_j \text{ for each $j \notin \fI_i$}\}.
$$
When we define the subset $J_i = \{\left\vert x_j \right\vert \colon j \in \fI_i\}$ of $I$ for each~$1 \leq i \leq \lfloor \frac{r-1}{2} \rfloor$, we have the following equivalence:
$$
\sum_{y \in \cT^x_i}{a_y}y = 0 \text{ if and only if } \sum_{y \in \cT^x_i}{a_y}\rho_{J_i}(y) = 0,
$$
where $a_y \in \{-1,0,1\}$.
From Example~\ref{example:K_3_I} and Example~\ref{ex_K2}, it follows that $\cH_5 \subset \ker \Psi_I$ when $r = 3$ and $\cH_2 \subset \ker \Psi_I$ when $r =4$, respectively.
Therefore, we conclude that $\cH_2 \cup \cH_5 \subset \ker \Psi_I$, which confirms Theorem~\ref{kernel}.

\begin{theorem}\label{kernel}
	The image $\Psi_I(M_I) = \{\Psi_I(x) \colon x \in M_I\}$ is trivial.
\end{theorem}

\section{$B$-snakes as elements of a basis of $H^\ast(X^\R_{B_n};\Q)$}\label{Sec:B-snakes}

Let $I$ be a subset of $[n] =\{1,\ldots,n\}$ with cardinality $r$.
For each $x \in \fS^B_I$, let $\bar{x}$ denote the signed permutation on $I$ defined by
\begin{equation}\label{barx}
  \bar{x}(i) = -x(i)
\end{equation}
for each $i \in I$.
Consider the signed permutations $x = x_r\cdots x_1$ and $y = y_r\cdots y_1$ on $I$.
We define a relation $x \prec y$ as follows; there is $1 \leq i \leq \lfloor \frac{r}{2} \rfloor$ such that
\begin{enumerate}
  \item $x_{2i-1} + x_{2i+1} < y_{2i-1} + y_{2i+1}$, and
  \item $x_{2j-1} + x_{2j+1} = y_{2j-1} + y_{2j+1}$ for all $1 \leq j < i$.
\end{enumerate}
In addition, define a relation $x \vartriangleleft y$ by
$$
\begin{cases}
  x \prec y, & \mbox{if $r$ is odd} \\
  \bar{x} \prec \bar{y}, & \mbox{if $r$ is even}.
\end{cases}
$$
Therefore, the relation $\trianglelefteq$ forms a partial order on $\fS^B_I$.

\begin{example}
    Let $I_1 = \{1,2,3,4,5\}$ and $I_2 = \{1,2,3,4,5,6\}$.
    \begin{enumerate}
      \item Consider signed permutations $x = [2/45/31]$, $y = [4/\bar{5}2/31]$, and $z = [4/21/\bar{5}3]$ on $I_1$.
      We find that $z \prec y \prec x$, and thus, $z \triangleleft y \triangleleft z$.
      \item For signed permutations $x = [\bar{1}\bar{6}/54/\bar{2}\bar{3}]$, $y = [61/54/\bar{2}\bar{3}]$, and $z = [61/5\bar{2}/4\bar{3}]$ on $I_2$, we compute
          $$
          \bar{x} = [16/\bar{5}\bar{4}/23], \bar{y} = [\bar{6}\bar{1}/\bar{5}\bar{4}/23], \bar{z} = [\bar{6}\bar{1}/\bar{5}2/\bar{4}3].
          $$
          From this, $\bar{z} \prec \bar{y} \prec \bar{x}$, and it follows that $z \triangleleft y \triangleleft x$.
    \end{enumerate}
\end{example}

A signed permutation $x_rx_{r-1}\cdots x_1 \in \fS^B_I$ is called a \emph{$B$-snake} on $I$ if it satisfies
$$
0 < x_r > x_{r-1} < \cdots x_1.
$$
The set of $B$-snakes on $I$ is denoted by $\fA^B_I$.
For $B$-snakes $x$ and $y$ on $I$, the following lemma holds.

\begin{lemma}\label{odd_order_lemma}
	Let $x= x_r\cdots x_1$ be a $B$-snake on $I$.
	If $y = y_r\cdots y_1 \in \fA_I^B$ and the simplex
	$$\{\cF_1(x),\cF_2(x),\ldots,\cF_{\lfloor \frac{r+1}{2} \rfloor}(x)\}$$
	is a face of $\square_y$, then $x \trianglelefteq y$.
\end{lemma}

\begin{proof}
  We immediately obtain that $x_1$ must be $y_1$ or $y_2$.
  For each $1 \leq i < \lfloor \frac{r}{2} \rfloor$, one of the following conditions holds:
  \begin{enumerate}
    \item $x_{2i} \in \{y_{2i-1},y_{2i}\},x_{2i+1} \in \{y_{2i+1},y_{2i+2}\}$, or
    \item $x_{2i+1} \in \{y_{2i-1},y_{2i}\},x_{2i} \in \{y_{2i+1},y_{2i+2}\}$.
  \end{enumerate}
  Then, either $y_1 + y_2 = x_1 + x_2$ or $y_1 + y_2 = x_1 + x_3$.
  
  \medskip

  \noindent \underline{\textbf{CASE 1. $y_1 + y_2 = x_1 + x_3$ :}}
  If $r$ is odd, then $x_3 > x_2$, and consequently, $x_1 + x_2 < x_1 + x_3 = y_1 + y_2$.
  If $r$ is even, then $x_3 < x_2$, which implies $\bar{x}_1 + \bar{x}_2 < \bar{x}_1 + \bar{x}_3 = \bar{y}_1 + \bar{y}_2$.
  In this case, we conclude that $x \triangleleft y$.
  
    \medskip

  \noindent \underline{\textbf{CASE 2. $y_1 + y_2 = x_1 + x_2$ :}}
  Let $i$ be the largest integer $1 \leq i \leq \lfloor \frac{r}{2} \rfloor$ such that
  $$
  x_{2j-1}+x_{2j} = y_{2j-1} + y_{2j} \text{ for all }1 \leq j \leq i.
  $$
  Since one of $y_1$ and $y_2$ is $x_1$, the other must be $x_2$.
  As $y$ is a $B$-snake on $I$, it follows that $y_1 = x_1$ and $y_2 = x_2$.
  By a similar argument, we conclude that $y_k = x_k$ for all $1 \leq k \leq 2i$.
  
  Consider the case where $i = \lfloor \frac{r}{2} \rfloor$.
  If $r$ is even, it is straightforward that  $x = y$.
  If $r$ is odd, then $\left\vert x_r \right\vert = \left\vert y_r \right\vert$.
  Since both $x$ and $y$ are $B$-snakes, both $x_r$ and $y_r$ are positive, \ie $x_r = y_r$.
  Thus, we conclude that $x = y$.
  
  Now, let us consider the case where $i < \lfloor \frac{r}{2} \rfloor$.
  Since $x_{2i} = y_{2i} \in \{y_{2i-1},y_{2i}\}$, we have $$x_{2i+1} \in \{y_{2i+1},y_{2i+2}\}.$$
  Furthermore, since $x_{2i+1} + x_{2i+2} \neq y_{2i+1}+y_{2i+2}$, it follows that $x_{2i+2} \notin \{y_{2i+1},y_{2i+2}\}$.
  Thus, it follows that
  $$
  x_{2i+3} \in \{y_{2i+1},y_{2i+2}\},
  $$
  and then, $y_{2i+1} + y_{2i+2} = x_{2i+1} + x_{2i+3}$.
  Note that
  $$
  \begin{cases}
    x_{2i+1} + x_{2i+3} > x_{2i+1} + x_{2i+2}, & \mbox{if $r$ is odd}, \\
    x_{2i+1} + x_{2i+3} < x_{2i+1} + x_{2i+2}, & \mbox{if $r$ is even}.
  \end{cases}
  $$
  Therefore, we have $x \triangleleft y$.
\end{proof}

Lemma~3.4 in \cite{Choi-Park-Park2017typeB} establishes a homotopy equivalence between $(K_{B_{n}})_{I}$ and the subcomplex~$\widehat{(K_{B_{n}})_{I}}$ of $(K_{B_{n}})_{I}$ which is obtained by iteratively removing vertices from $(K_{B_{n}})_{I}$ that is not a subset of $I \cup \{-i \colon i \in I\}$.
Since $\square_x \in \widehat{(K_{B_{n}})_{I}}$ for all $x \in \fA^B_I$, to prove Lemma~\ref{basis_B}, it suffices to demonstrate that $\Psi_I(\fA^B_I)$ forms a $\Q$-basis of $\widetilde{H}_{\lfloor \frac{r-1}{2} \rfloor}(\widehat{(K_{B_{n}})_{I}})$.

\begin{lemma}\label{basis_B}
Let $I$ be a subset of $[n]$ with cardinality $r$.
The image $\Psi_I(\fA^B_I) = \{\Psi_{I}(x) \colon x \in \fA^B_I\}$ is a $\Q$-basis of $\widetilde{H}_{\lfloor \frac{r-1}{2} \rfloor}((K_{B_{n}})_{I})$.
\end{lemma}

\begin{proof}
Since the dimension of $\widehat{(K_{B_{n}})_{I}}$ is $\lfloor \frac{r-1}{2} \rfloor$, every element of $\widetilde{H}_{\lfloor \frac{r-1}{2} \rfloor}(\widehat{(K_{B_{n}})_{I}})$ can be regarded as a cycle of $C_{\lfloor \frac{r-1}{2} \rfloor}(\widehat{(K_{B_{n}})_{I}})$.
Assume that a $\Q$-linear combination $\cX$ of $\{\square_x \colon x \in \fA^B_I\}$ is zero.
	Let $\mathfrak{M}$ denote the set of maximal elements of $(\fA^B_I,\prec)$.
	By Lemma~\ref{odd_order_lemma}, for each $x \in \mathfrak{M}$, the coefficient of $x$ in $\cX$ should be zero for all $x \in \mathfrak{M}$.
	Repeating this argument for the maximal elements of $(\fA^B_I \setminus \mathfrak{M} , \prec)$, and continuing inductively, we find that all coefficients in $\cX$ are zero.
	Hence, we have $\Psi_I(\fA^B_I)$ is $\Q$-linearly independent.
	Combining this with~\eqref{typeB_I}, we conclude that $\Psi_I(\fA^B_I)$ is a $\Q$-basis of $\widetilde{H}_{\lfloor \frac{r-1}{2} \rfloor}(\widehat{(K_{B_{n}})_{I}})$.
\end{proof}

Theorem~\ref{kernel} establishes the well-definedness of the composition
\begin{equation}\label{map}
    \Q\left\langle \fS^B_I \right\rangle / M_I \overset{\pi^{-1}}{\longrightarrow} \Q\left\langle \fS^B_I \right\rangle \overset{\Psi_I}{\longrightarrow} \widetilde{H}_{\lfloor \frac{r-1}{2} \rfloor}((K_{B_{n}})_{I}) \overset{\ast}{\longrightarrow} \widetilde{H}^{\lfloor \frac{r-1}{2} \rfloor}((K_{B_{n}})_{I}),
\end{equation}
	where $\pi^{-1}(x + M_I) = x$ for each $x \in \fS^B_I$ and $\ast$ is the canonical isomorphism to the dual space.
	In addition, Lemma~\ref{basis_B} confirms the surjectivity of this composition.
Referring to \eqref{typeB_I}, to complete the proof of Theorem~\ref{theorem_in4}, it remains to show that every signed permutation on $I$ can be written as a linear combination of $B$-snakes in $\Q\left\langle \fS^B_I \right\rangle /M_I$.

\begin{theorem}~\label{theorem_in4}
    For a subset $I \subset [n]$ of cardinality $r$,
    $$
    \Q\left\langle \fA^B_I \right\rangle \overset{\hookrightarrow}{\cong} \Q\left\langle \fS^B_I \right\rangle / M_I \overset{\Psi_I \circ \pi^{-1}}{\cong} \widetilde{H}^{\lfloor \frac{r-1}{2} \rfloor}((K_{B_{n}})_{I})
    $$
    as a $\Q$-vector space.
\end{theorem}

\begin{proof}
Let us consider the generator sets $\cH_1$ and $\cH_3$ of $M_I$, defined in Definition~\ref{def M_I}.
Then, we may assume that each signed permutation $x = x_r\cdots x_2x_1$ in $\Q\left\langle \fS^B_I \right\rangle /M_I$ satisfies the following conditions:
	$$
	\begin{cases}
		x_r>0 \text{ and }x_{2i-1} < x_{2i}, &\mbox{if $r$ is odd},\\
		x_{2i-1} > x_{2i}, &\mbox{if $r$ is even},
	\end{cases}
	$$
	for each $1 \leq i \leq \lfloor \frac{r}{2} \rfloor$.
    
    For a signed permutation $x = x_r\cdots x_2 x_1 \in \Q\left\langle \fS^B_I \right\rangle /M_I$, if $x$ is not a $B$-snake, then at least one of the following conditions must hold:
    \begin{enumerate}
      \item There is $1 \leq i < \lfloor \frac{r}{2} \rfloor$ such that $x_{2i+2},x_{2i+1},x_{2i},x_{2i-1}$ are monotone.
      \item $r$ is even and $x_r < 0$.
      \item $r \geq 3$ is odd and $x_r < x_{r-1}$.
    \end{enumerate}
    If any of these conditions holds for $x$, then by the relation $\cH_2$, $\cH_4$, and $\cH_5$ in Definition~\ref{def M_I}, one can verify that $x + M_I$ can be expressed by a linear sum of $(y +M_I)$s satisfying $y \triangleleft x$, where $y \in \fS^B_I$.
    This procedure is carried out recursively and must terminate in finitely many steps, which proves the theorem.
    
\end{proof}

From Theorem~\ref{theorem_in4}, every element $x \in \fS^B_I$ is written as $\alpha \in \fA_I^B$ within $\Q\left\langle \fS^B_I  \right\rangle/ M_I$.
We denote by $\cC^{\alpha}_x$ the coefficient of $\alpha \in \fA^B_I$ in the expression of~$x$, as
$$
	x = \sum_{\alpha \in \fA^B_I}\cC^{\alpha}_x\cdot\alpha.
$$
After conducting calculations for small positive integers~$r$, the author was able to confirm that every $\cC^\alpha_x$ is $-1$, $0$, or $1$.
However, it is unclear whether nonzero $\cC^\alpha_x$ always takes $\pm 1$.

\section{Cohomology rings of $X^\R_{B_n}$}\label{Sec:multi}
 
Let $I_1$ and $I_2$ be subsets of $[n]$.
We assume that $I_1 \cap I_2 = \emptyset$ and $\left\vert I_1 \right\vert \cdot \left\vert I_2 \right\vert$ is even.
A $B$-snake $z = z_{\ell}\cdots z_2z_1 \in \fA^B_{I_1 \cup I_2}$ is said to be \emph{restrictable} to $(I_1,I_2)$, if for all $1 \leq i \leq \lfloor \frac{\ell-1}{2} \rfloor$, either
$$
\{\left\vert z_{2i-1} \right\vert,\left\vert z_{2i} \right\vert\} \subset I_1  \text{ or } \{\left\vert z_{2i-1} \right\vert,\left\vert z_{2i} \right\vert\} \subset I_2.
$$
The set of restrictable $B$-snakes to $(I_1,I_2)$ is denoted by $\cR^B_{(I_1,I_2)}$.
For each $z = z_\ell\cdots z_1 \in \cR^B_{(I_1,I_2)}$, let $\kappa_{(I_1,I_2)}(z)$ denote the number of pairs $(z_{2i-1},z_{2j-1})$ satisfying
	$$
	\left\vert z_{2i-1} \right\vert \in I_1, \left\vert z_{2j-1} \right\vert \in I_2,  \text{ and } \lfloor \frac{\ell+1}{2} \rfloor \geq i > j \geq 1. 
	$$

Let $I$ be a subset of $[n]$, and $x$ a signed permutation on $I$.
For each $J \subset I$, define a signed permutation $P_J(x) \in \fS^B_J$ as follows:
	$$
    P_J(x) =
	\begin{cases}
		\rho_J(x), & \mbox{if $\left\vert I \right\vert + \left\vert J \right\vert$ is even}, \\
		\rho_J(\bar{x}), & \mbox{if $\left\vert I \right\vert + \left\vert J \right\vert$ is odd},
	\end{cases}
	$$
where $\bar{x}$ is defined in~\eqref{barx}, and $\rho_J(x) \in \fS_J^B$ is a subpermutation of $x$ restricted to $J$.

\begin{example}\label{example:5.1}
Let us consider subsets $I_1 = \{1,4\}$ and $I_2 = \{2,3\}$ of $I = \{1,2,3,4\}$.
The $B$-snake~$[1\bar{3}/{4}2]$ is not restrictable to $(I_1,I_2)$.
On the other hand, $B$-snakes $[1\bar{4}/32]$ and $[1\bar{4}/\bar{2}\bar{3}]$ are restrictable to $(I_1,I_2)$.
Moreover, we have
	$$
	P_{I_1}([1\bar{4}/32]) = P_{I_1}([1\bar{4}/\bar{2}\bar{3}]) = [1\bar{4}], P_{I_2}([1\bar{4}/32]) = [32], \text{ and } P_{I_2}([1\bar{4}/\bar{2}\bar{3}]) = [\bar{2}\bar{3}].
	$$
Considering that $2,3 \in I_2$ and $4 \in I_1$, it follows that
$$
\kappa_{(I_1,I_2)}([1\bar{4}/32]) = \kappa_{(I_1,I_2)}([1\bar{4}/\bar{2}\bar{3}]) = 1.
$$
\end{example}

\begin{example}\label{example:5.2}
Consider subsets $I_1 = \{1,5\}$ and $I_2 = \{2,3,4\}$ of $I = \{1,2,3,4,5\}$.
We have that $[2/\bar{1}3/\bar{4}5] \notin \cR^B_{(I_1,I_2)}$ and $z = [z_5/z_4z_3/z_2z_1] = [2/15/\bar{4}\bar{3}] \in \cR^B_{(I_1,I_2)}$.
	We obtain
	$$
	P_{I_1}(z) = [\bar{1}\bar{5}], P_{I_2}(z) = [2/\bar{4}\bar{3}].
	$$
Note that
\begin{enumerate}
  \item $\left\vert z_3 \right\vert = 5 \in I_1$, $\left\vert z_1 \right\vert =3 \in I_2$, 
  \item $\left\vert z_5 \right\vert = 2 \in I_2$, $\left\vert z_1 \right\vert =3 \in I_2$, and
  \item $\left\vert z_5 \right\vert = 2 \in I_2$, $\left\vert z_3 \right\vert =5 \in I_1$.
\end{enumerate}
Thus, $\kappa_{(I_1,I_2)}(z) = \left\vert \{(z_{2i-1},z_{2j-1}) \colon \left\vert z_{2i-1} \right\vert \in I_2, \left\vert z_{2j-1} \right\vert \in I_1, 1 \leq i < j \leq 3\}\right\vert = \left\vert \{(5,-3)\}\right\vert = 1$.
\end{example}

We now denote $(I_1 \cup I_2) \setminus (I_1 \cap I_2)$ by $I_1 \triangle I_2$.

\begin{definition}\label{def-multi2}
	Let  $\smile \colon \Q\left\langle \fA^B_{I_1}  \right\rangle \otimes \Q\left\langle \fA^B_{I_2}  \right\rangle \to \Q\left\langle \fA^B_{I_1 \triangle I_2}  \right\rangle$ be the multiplicative structure of the $\Q$-vector algebra
$$
\bigoplus_{I \subset [n]} \Q\left\langle \fA^B_{I}  \right\rangle
$$
defined by
	$$
	\alpha \smile \beta = \begin{cases}
		\underset{z \in \cR^B_{(I_1,I_2)}}{\sum}(-1)^{\kappa_{(I_1,I_2)}(z)} \cC^\alpha_{P_{I_1}(z)} \cC^\beta_{P_{I_2}(z)} \cdot z, & \mbox{if } I_1 \cap I_2 = \emptyset, \left\vert I_1 \right\vert \cdot \left\vert I_2 \right\vert \text{ is even},\\
		0, & \mbox{otherwise,}
	\end{cases}
	$$
for $\alpha \in \fA^B_{I_1}$ and $\beta \in \fA^B_{I_2}$.
	
\end{definition}

\begin{example}
    Let $I_1 = \{1,4\}$, $I_2 = \{2,3\}$ be subsets of $\{1,2,3,4\}$.
    We consider the multiplications $[1\bar{4}] \smile [32]$ and $[41] \smile [3\bar{2}]$.
    There are exactly $20$~restrictable $B$-snakes to $(I_1,I_2)$.
    Refer Table~\ref{table1}.
    \begin{table}
    \renewcommand{\arraystretch}{1.5} 
      \centering
      \begin{tabular}{c|c|c|c|c}
            \hline
		\diagbox{$P_{I_1}(z)$}{$P_{I_2}(z)$} & $ [32]$ & $[3\bar{2}]$ & $[2\bar{3}]$ & $[\bar{2}\bar{3}]$ \\
\hline
		$[41]$  & $[41/32],[32/41]$ & $[41/3\bar{2}],[3\bar{2}/41]$ & $[41/2\bar{3}],[2\bar{3}/41]$ &  \\
\hline
		$[4\bar{1}]$  & $[4\bar{1}/32],[32/4\bar{1}]$ & $[4\bar{1}/3\bar{2}],[3\bar{2}/4\bar{1}]$ & $[4\bar{1}/2\bar{3}],[2\bar{3}/4\bar{1}]$ &\\
\hline
		$[1\bar{4}]$ & $[1\bar{4}/32]$ & $[1\bar{4}/3\bar{2}],[3\bar{2}/1\bar{4}]$ & $[1\bar{4}/2\bar{3}],[2\bar{3}/1\bar{4}]$& $[1\bar{4}/\bar{2}\bar{3}]$ \\
\hline
		$[\bar{1}\bar{4}]$ &  & $[3\bar{2}/\bar{1}\bar{4}]$ & $[2\bar{3}/\bar{1}\bar{4}]$ &\\
\hline
	\end{tabular}
      \caption{The restrictable $B$-snakes}\label{table1}
    \end{table}
      We consider $z \in \cR^B_{(I_1,I_2)}$ such that neither $\cC^{[1\bar{4}]}_{P_{I_1}(z)}$ nor $\cC^{[32]}_{P_{I_2}(z)}$ is zero.
      Note that, for each pair~$(x,y)$ of $B$-snakes on a finite set, the coefficient $\cC^y_x$ is the kronecker delta~$\delta_{x,y}$.
      Then we obtain that either $P_{I_1}(z)$ is $[1\bar{4}]$ or $[\bar{1}\bar{4}]$, and $P_{I_2}(z)$ is $[32]$ or $[\bar{2}\bar{3}]$.
      Hence, 
      \begin{equation}\label{ex54}
        z = [1\bar{4}/32], \text{ or } [1\bar{4}/\bar{2}\bar{3}].
      \end{equation}
      By Definition~\ref{def M_I},
      $$
      [\bar{2}\bar{3}] + M_{I_2} = [32] - [3\bar{2}] + [2\bar{3}] + M_{I_2}.
      $$
      Then we have $\cC^{[32]}_{[\bar{2}\bar{3}]} =1$.
      Combining Example~\ref{example:5.1} with \eqref{ex54}, we conclude that
      \begin{align*}
        [1\bar{4}] \smile [32] & =(-1)^{\kappa_{(I_1,I_2)}([1\bar{4}/32])}\cdot[1\bar{4}/32] + (-1)^{\kappa_{(I_1,I_2)}([1\bar{4}/\bar{2}\bar{3}])}\cdot[1\bar{4}/\bar{2}\bar{3}]. \\
         & = -[1\bar{4}/32] - [1\bar{4}/\bar{2}\bar{3}].
      \end{align*}
      
      Similarly, for $z \in \cR_{(I_1,I_2)}^B$, if $\cC^{[41]}_{P_{I_1}(z)}\cC^{[3\bar{2}]}_{P_{I_2}(z)} \neq 0$, it follows that $z = [51/3\bar{2}]$, $[3\bar{2}/41]$ or $[3\bar{2}/\bar{1}\bar{4}]$.
      One can easily deduce that
      $$
      [41] \smile [3\bar{2}] = -[41/3\bar{2}] +[3\bar{2}/41] + [3\bar{2}/\bar{1}\bar{4}].
      $$
      
\end{example}

\begin{example}
Let $I_1 = \{1,5\}$ and $I_2 = \{2,3,4\}$ be subsets of $\{1,2,3,4,5\}$.
To compute the multiplication $[51] \smile [2/\bar{4}\bar{3}]$ of $[51] \in \fA^B_{I_1}$ and $[2/\bar{4}\bar{3}] \in \fA^B_{I_2}$,
we consider only the restrictable $B$-snakes $z \in \cR_{(I_1,I_2)}^B$ that satisfy the following two conditions:
  \begin{enumerate}
    \item $P_{I_1}(z) = [51]$ or $[\bar{1}\bar{5}]$.
    \item $P_{I_2}(z) = [2/\bar{4}\bar{3}]$ or $[2/34]$.
  \end{enumerate}
  By some calculations, we conclude that
$$
    [51] \smile [2/\bar{4}\bar{3}] =  -[2/15/\bar{4}\bar{3}] -[2/\bar{5}\bar{1}/\bar{4}\bar{3}]- [2/15/34]+ [2/\bar{4}\bar{3}/\bar{5}\bar{1}].
$$
\end{example}

We now focus on the multiplicative structure of $H^\ast(X^{\R}_{B_n};\Q)$.
Combining \eqref{map} with Theorem~\ref{ChoiPark}, we obtain an extended $\Q$-vector space isomorphism
$$
	\bigoplus_{I \subset [n]} \Q\left\langle \fS^B_I  \right\rangle/M_I \to H^\ast(X^{\R}_{B_n};\Q).
$$
By Theorem~\ref{kernel}, the above map can be expressed as
$$
\boldsymbol{\Psi} \colon \bigoplus_{I \subset [n]} \Q\left\langle \fA^B_I  \right\rangle \to H^\ast(X^{\R}_{B_n};\Q).
$$
Here, we aim to demonstrate that the $\Q$-vector space isomorphism $\boldsymbol{\Psi}$ is also a $\Q$-algebra isomorphism.

We assume that $I_1$ and $I_2$ are subsets of $[n]$ of cardinality $r$ and $s$, respectively.
Consider the homomorphism
$$
	\iota_{\ast} \colon \widetilde{H}_{\lfloor \frac{r+1}{2} \rfloor + \lfloor \frac{s+1}{2} \rfloor -1}((K_{B_n})_{I_1 \triangle I_2}) \to \widetilde{H}_{\lfloor \frac{r+1}{2} \rfloor + \lfloor \frac{s+1}{2} \rfloor -1}((K_{B_n})_{I_1} \star (K_{B_n})_{I_2}),
$$
induced by the simplicial map $\iota \colon (K_{B_n})_{I_1 \triangle I_2} \hookrightarrow (K_{B_n})_{I_1} \star (K_{B_n})_{I_2}$.

\begin{lemma}\label{lemma2}
If $I_1 \cap I_2 = \emptyset$ and $\left\vert I_1 \right\vert \cdot \left\vert I_2 \right\vert$ is even, then
	$$
	\iota_{\ast}(\Psi_{I_1 \cup I_2}(z)) =
	\begin{cases}
		(-1)^{\kappa_{(I_1,I_2)}(z)}\Psi_{I_1}(P_{I_1}(z))\star\Psi_{I_2}(P_{I_2}(z)), & \mbox{for }z \in \cR^B_{(I_1,I_2)}, \\
		0, & \mbox{for }z \notin \cR^B_{(I_1,I_2)}.
	\end{cases}
	$$
\end{lemma}

\begin{proof}
	To begin, let us consider the case where $z \notin \cR^B_{(I_1,I_2)}$.
There exists $1 \leq p \leq \lfloor \frac{r+s}{2} \rfloor$ such that neither $I_1$ nor $I_2$ includes $\{\left\vert z_{2p-1} \right\vert, \left\vert z_{2p} \right\vert\}$.
We choose $p$ to be the largest possible value.
Then, $\{\cF_{p}(z),\cF_{p}(z^\ast)\}$ is a line segment in $(K_{B_n})_{I_1} \star (K_{B_n})_{I_2}$.
Consider the subpermutation
$$
\hat{z} \in \fS^B_{(I_1\cup I_2) \setminus \{\left\vert z_{2p-1} \right\vert,\left\vert z_{2p} \right\vert\}}
$$
of $z$.
It follows that
$$
\Psi_{I_1 \cup I_2}(z) = [\pm \square_{\hat{z}} \star \{\cF_{p}(z),\cF_{p}(z^\ast)\}]= [\pm \partial( \square_{\hat{z}} \star \{\{\cF_{p}(z),\cF_{p}(z^\ast)\}\})] = 0.
$$

We now assume that $z = z_{r+s}\cdots z_2z_1 \in \cR^B_{(I_1,I_2)}$.
Let $i_1 < \cdots < i_{\lfloor \frac{r-1}{2} \rfloor}$ be the indices such that $z_{2i_k-1} \in I_1$ for $1 \leq k \leq {\lfloor \frac{r-1}{2} \rfloor}$, and $j_1 < \cdots < j_{\lfloor \frac{s-1}{2} \rfloor}$ be the indices such that $z_{2j_\ell-1} \in I_2$ for $1 \leq \ell \leq {\lfloor \frac{s-1}{2} \rfloor}$.
Then, we have
$$
\begin{cases}
   \Psi_{I_1}(P_{I_1}(z)) = [\{\cF_{i_1}(z),\cF_{i_1}(z^\ast)\}\star \cdots \star \{\cF_{i_{\lfloor \frac{r-1}{2} \rfloor}}(z),\cF_{i_{\lfloor \frac{r-1}{2} \rfloor}}(z^\ast)\}] \in H_{\lfloor \frac{r-1}{2} \rfloor}((K_{B_n})_{I_1}), \\
  \Psi_{I_2}(P_{I_2}(z)) = [\{\cF_{j_1}(z),\cF_{j_1}(z^\ast)\}\star \cdots \star \{\cF_{j_{\lfloor \frac{s-1}{2} \rfloor}}(z),\cF_{j_{\lfloor \frac{s-1}{2} \rfloor}}(z^\ast)\}] \in H_{\lfloor \frac{s-1}{2} \rfloor}((K_{B_n})_{I_2}).
\end{cases}
$$
Therefore, $\iota_{\ast}(\Psi_{I_1 \cup I_2}(z)) = (-1)^{\kappa_{(I_1,I_2)}(z)}\Psi_{I_1}(P_{I_1}(z)) \star \Psi_{I_2}(P_{I_2}(z))$.
\end{proof}

\begin{theorem}\label{thm:final}
    A $\Q$-linear map
    $$
    \boldsymbol{\Psi} \colon (\bigoplus_{I \subset [n]} \Q\left\langle \fA^B_I  \right\rangle ,+,\smile) \to H^\ast(X^{\R}_{B_n};\Q).
    $$
    is a $\Q$-algebra isomorphism.
\end{theorem}

\begin{proof}
	It is sufficient to prove that
\begin{equation}\label{5.1}
(\boldsymbol{\Psi}(\alpha) \smile \boldsymbol{\Psi}(\beta))(\Psi_{I_1 \triangle I_2}(z)) = \boldsymbol{\Psi}(\alpha \smile \beta)(\Psi_{I_1 \triangle I_2}(z))
\end{equation}
for each $\alpha \in \fA^B_{I_1},$
	$\beta \in \fA^B_{I_2}$, and $z \in \fS^B_{I_1 \triangle I_2}$.

If $I_1$ and $I_2$ satisfy the following two conditions:
\begin{enumerate}
  \item $I_1 \cap I_2 \neq \emptyset$,
  \item $I_1 \cap I_2 = \emptyset$ and $\left\vert I_1 \right\vert \cdot \left\vert I_2 \right\vert$ is odd.
\end{enumerate}
Then $(K_{B_n})_{I_1 \triangle I_2}$ is homotopy equivalent to a bouquet of spheres with dimension at most $\lfloor \frac{r+1}{2} \rfloor + \lfloor \frac{s+1}{2} \rfloor -1$.
It follows that
$$
\widetilde{H}^{\lfloor \frac{r+1}{2} \rfloor + \lfloor \frac{s+1}{2} \rfloor -1}((K_{B_n})_{I_1 \triangle I_2}) =0.
$$
In conclusion, the following map is trivial:
$$
	\widetilde{H}^{\lfloor \frac{r-1}{2} \rfloor}((K_{B_n})_{I_1}) \otimes \widetilde{H}^{\lfloor \frac{s-1}{2} \rfloor}((K_{B_n})_{I_2}) \to \widetilde{H}^{\lfloor \frac{r+1}{2} \rfloor + \lfloor \frac{s+1}{2} \rfloor -1}((K_{B_n})_{I_1 \triangle I_2}),
$$
which implies that the multiplicative structure of
$$
\widetilde{H}^{\lfloor \frac{r-1}{2} \rfloor}((K_{B_n})_{I_1}) \text{ and }\widetilde{H}^{\lfloor \frac{s-1}{2} \rfloor}((K_{B_n})_{I_2})
$$
is trivial, \ie $\boldsymbol{\Psi}(\alpha) \smile \boldsymbol{\Psi}(\beta) =0$.
Since $\alpha \smile \beta$ is zero in this case, the proof is complete in the case where $I_1 \cap I_2 \neq \emptyset$.
	
Consider the case where $I_1 \cap I_2 = \emptyset$ and $\left\vert I_1 \right\vert \cdot \left\vert I_2 \right\vert$ is even.
By Lemma~\ref{lemma2},
	$$
	\iota_{\ast}(\Psi_{I_1 \cup I_2}(z)) =
	\begin{cases}
		(-1)^{\kappa_{(I_1,I_2)}(z)}\Psi_{I_1}(P_{I_1}(z))\star\Psi_{I_2}(P_{I_2}(z)), & \mbox{for }z \in \cR^B_{(I_1,I_2)}, \\
		0, & \mbox{for }z \notin \cR^B_{(I_1,I_2)}.
	\end{cases}
	$$
	Thus, if $z \in \cR^B_{(I_1,I_2)}$, then we have
	\begin{align*}
		(\boldsymbol{\Psi}(\alpha) \smile \boldsymbol{\Psi}(\beta))(\Psi_{I_1 \cup I_2}(z)) &= (\boldsymbol{\Psi}(\alpha) \otimes \boldsymbol{\Psi}(\beta))(\iota_{\ast}(\Psi_{I_1 \cup I_2}(z)))\\
		&= (-1)^{\kappa_{(I_1,I_2)}(z)}\boldsymbol{\Psi}(\alpha)(\Psi_{I_1}(P_{I_1}(z))) \boldsymbol{\Psi}(\beta)(\Psi_{I_2}(P_{I_2}(z)))\\
		&= (-1)^{\kappa_{(I_1,I_2)}(z)} \cC^\alpha_{P_{I_1}(z)}\cC^\beta_{P_{I_2}(z)}.
	\end{align*}
	If $z \notin \cR^B_{(I_1,I_2)}$, then $(\boldsymbol{\Psi}(\alpha) \smile \boldsymbol{\Psi}(\beta))(\Psi_{I_1 \cup I_2}(z)) = 0.$
	Furthermore,
	\begin{align*}
		\boldsymbol{\Psi}(\alpha \smile \beta)(\Psi_{I_1 \cup I_2}(z)) &= \boldsymbol{\Psi}(\sum_{z \in \cR^B_{(I_1,I_2)}}(-1)^{\kappa_{(I_1,I_2)}(z)}\cC^\alpha_{P_{I_1}(z)}\cC^\beta_{P_{I_2}(z)} \cdot z)(\Psi_{I_1 \cup I_2}(z)) \\
		&= \begin{cases}
			(-1)^{\kappa_{(I_1,I_2)}(z)}\cC^\alpha_{P_{I_1}(z)}\cC^\beta_{P_{I_2}(z)}, & \mbox{for }z \in \cR^B_{(I_1,I_2)}, \\
			0, & \mbox{for }z \notin \cR^B_{(I_1,I_2)},
		\end{cases}
	\end{align*}
which confirms this theorem.
\end{proof}

\end{document}